\documentclass[11pt,a4paper]{amsart}
\usepackage{amssymb,amsmath}

\usepackage{times}
\usepackage{hyperref}

\numberwithin{equation}{section}
\newtheorem{theorem}{Theorem}[section]
\newtheorem{lemma}[theorem]{Lemma}
\newtheorem{proposition}[theorem]{Proposition}
\newtheorem{corollary}[theorem]{Corollary}

\theoremstyle{definition}
\newtheorem{definition}[theorem]{Definition}
\newtheorem{example}[theorem]{Example}

\newcommand\Supp{\operatorname{Supp}}

\newcommand\Hom{\operatorname{Hom}}

\newcommand\Ext{\operatorname{Ext}}
\newcommand\Rad{\operatorname{Rad}}
\newcommand\Ker{\operatorname{Ker}}

\newcommand\im{\operatorname{Im}}

\newcommand\cd{\operatorname{cd}}

\newcommand\height{\operatorname{height}}
\newcommand\grade{\operatorname{grade}}

\newcommand\Spec{\operatorname{Spec}}
\newcommand{\gam}{\Gamma_{I}}
\newcommand{\Rgam}{{\rm R} \Gamma_{\mathfrak m}}

\newcommand{\qism}{\stackrel{\sim}{\longrightarrow}}

\begin{document}
\author[P. Schenzel]{Peter Schenzel}
\title[Dimensions of local cohomology]{On endomorphism rings and dimensions of local cohomology
modules}


\address{Martin-Luther-Universit\"at Halle-Wittenberg,
Institut f\"ur Informatik, D --- 06 099 Halle (Saale),
Germany}

\email{peter.schenzel@informatik.uni-halle.de}



\begin{abstract}
Let $(R,\mathfrak m)$ denote an $n$-dimensional complete local Gorenstein ring.
For an ideal $I$ of $R$ let $H^i_I(R), i \in \mathbb Z,$ denote the local
cohomology modules of $R$ with respect to $I.$ If $H^i_I(R) = 0$ for all
$i \not= c = \height I,$ then the endomorphism ring of $H^c_I(R)$ is isomorphic to
$R$ (cf. \cite{HSt} and \cite{HS}). Here we prove that this is true if and only
if $H^i_I(R) = 0, i = n, n -1$ provided $c \geq 2$ and $R/I$ has an isolated
singularity resp. if $I$ is set-theoretically a complete intersection in codimension
at most one. Moreover, there is a vanishing result of $H^i_I(R)$  for all
$i > m, m$ a given integer, resp. an estimate of the dimension of $H^i_I(R).$
\end{abstract}

\maketitle

\section{Main Results}

Let $(R, \mathfrak m)$ denote a local Noetherian ring with $n = \dim R.$ For the ideal
$I \subset R$ let $H^i_I(\cdot), i \in \mathbb Z,$ denote the local cohomology
functor with respect to $I,$ see \cite{aG} for its definition and basic results. It
is a difficult question to describe  $\sup \{ i \in \mathbb Z | H^i_I(R) \not= 0\},$
the cohomological dimension $\cd I$ of $I$ with respect to $R.$ Recall that $\height I
\leq \cd I.$ Recently there are some interesting results for ideals with $ c = \height I =
\cd I,$ so called cohomological complete intersections. If $(R, \mathfrak m)$ is a
complete local ring Hellus and St\"uckrad (cf. \cite{HSt}) have shown that the endomorphism
ring $\Hom_R(H^c_I(R), H^c_I(R))$ is isomorphic to $R.$ See also \cite[Lemma 2.8]{HS} for a
more functorial proof and a slight extension in the case of $(R,\mathfrak m)$ a Gorenstein ring.

The first aim of the consideration here is a characterization when the endomorphism ring
of $H^c_I(R)$ is isomorphic to $R.$ To this end we call $I$ locally a cohomological
complete intersection provided $\cd IR_{\mathfrak p} = \height I$ for all
prime ideals $\mathfrak p \in V(I)\setminus \{\mathfrak m\}.$ For instance, if $I$ has an
isolated singularity it is locally a cohomological complete intersection.

\begin{theorem} \label{1.1} Let $(R, \mathfrak m)$ denote a complete local Gorenstein
ring with $n = \dim R.$ Let $I$ be an ideal of $\height I = c.$ Suppose
that $I$ is locally a cohomological complete intersection. Then the following conditions are
equivalent:
\begin{itemize}
\item[(i)] The natural homomorphism $R \to \Hom_R(H^c_I(R), H^c_I(R))$ is an isomorphism.
\item[(ii)]  $H^i_I(R) = 0$ for $i = n-1, n.$
\end{itemize}
\end{theorem}

In a certain sense, condition (ii) of Theorem \ref{1.1} provides a numerical condition for the
property that the endomorphism ring of $H^c_I(R)$ is $R.$ In the case of $R$ a regular local
ring containing a field Huneke and Lyubeznik (cf. \cite[Theorem 2.9]{HL}) have shown a
topological characterization of the above condition (ii).

\begin{theorem} \label{1.2} Let $(R,\mathfrak m)$ be a local Gorenstein ring. Let $J \subset I$
denote two ideals of height $c.$
\begin{itemize}
\item[(a)] There is a natural homomorphism
\[
\Hom_R(H^c_J(R), H^c_J(R)) \to \Hom_R(H^c_I(R), H^c_I(R)).
\]
\item[(b)] Suppose that $\Rad J R_{\mathfrak p} = \Rad I R_{\mathfrak p}$ for all
$\mathfrak p \in V(I)$ with $\dim R_{\mathfrak p} \leq c + 1.$ Then the homomorphism in (a)
is an isomorphism.
\item[(c)] Let $R$ be in addition complete. Let $J$ denote a cohomologically complete
intersection contained in $I$ and satisfying the assumptions of (b). Then $R \to \Hom_R(H^c_I(R), H^c_I(R))$
is an isomorphism.
\end{itemize}
\end{theorem}

Our results are based on a certain estimate of $\dim H^i_I(R), i > c,$ see Theorem \ref{3.1}.
In the case of a regular local ring partial results of this type have been used by
Ken-Ichiroh Kawasaki (cf. \cite{kK}) for the study of Lyubeznik numbers (cf. \cite{gL}
for their definition). Here we use the truncation complex as invented in \cite[Section 2]{HS}
(cf. \ref{2.1}). Moreover it provides some technical statements about the endomorphism ring
of $H^c_I(R), c = \height I,$ (cf. Lemma \ref{2.2}).

In the terminology the author follows the paper \cite{HS}.

\section{On the truncation complex}
Let $(R, \mathfrak m, k)$ denote a local Gorenstein ring with $n = \dim R.$
First of all we will recall the truncation complex as it was introduced in \cite[Section 2]{HS}
and in a different context in \cite[\S 4]{pS2}.
Let $R \qism E^{\cdot}$ denote a minimal injective resolution of $R$ as an $R$-module. It is a
well-known fact that
\[
E^i \simeq \oplus_{\mathfrak p \in \Spec R, \height \mathfrak p = i} E_R(R/\mathfrak p),
\]
where $E_R(R/\mathfrak p)$ denotes the injective hull of $R/\mathfrak p$ (cf. \cite{hB} for
these and related results about Gorenstein rings).

Now let $I \subset R$ denote an ideal and $c = \height I.$ Then $d = \dim R/I = n - c.$ The local
cohomology modules $H^i_I(R), i \in \mathbb Z,$ are -- by definition -- the cohomology modules of
the complex $\gam (E^{\cdot}).$ Because of $\gam(E_R(R/\mathfrak p)) = 0$ for all
$\mathfrak p \not\in V(I)$ it follows that $\gam(E^{\cdot})^i = 0$ for all $i < c.$ Therefore
$H^c_I(R) = \Ker(\gam(E^{\cdot})^c \to \gam(E^{\cdot})^{c+1}).$ This observation provides an
embedding $H^c_I(R)[-c] \to \gam(E^{\cdot})$ of complexes of $R$-modules.

\begin{definition} \label{2.1} The cokernel of the embedding  $H^c_I(R)[-c] \to \gam (E^{\cdot})$
is defined as $C^{\cdot}_R(I),$ the truncation complex with respect to $I.$ So there is a short
exact sequence of complexes of $R$-modules
\[
0 \to H^c_I(R)[-c] \to \gam (E^{\cdot}) \to C^{\cdot}_R(I) \to 0.
\]
In particular it follows that $H^i(C^{\cdot}_R(I)) = 0$ for $i \leq c$ or $i > n$ and
$H^i(C^{\cdot}_R(I)) \simeq H^i_I(R)$ for $c < i \leq n.$
\end{definition}

For the first we need some basic results about the truncation complex. For more details we refer
to the exposition in \cite[Section 2]{HS}.

\begin{lemma} \label{2.2} With the previous notation there are the following results:
\begin{itemize}
\item[(a)] There are an exact sequence
 \[
0 \to H^{n-1}_{\mathfrak m}(C^{\cdot}_R(I)) \to H^d_{\mathfrak m}(H^c_I(R)) \to E
\to H^n_{\mathfrak m}(C^{\cdot}_R(I)) \to 0,
 \]
and isomorphisms $H^{i-c}_{\mathfrak m}(H^c_I(R)) \simeq H^{i-1}_{\mathfrak m}(C^{\cdot}_R(I))$
for $i < n.$
\item[(b)] $H^d_{\mathfrak m}(H^c_I(R)) \not= 0$ and $H^{i-c}_{\mathfrak m}(H^c_I(R)) = 0$
for $i > n.$
\item[(c)] Let $\mathfrak p \in V(I)$ denote a prime ideal. Then there is an isomorphism
\[
C^{\cdot}_R(I) \otimes_R R_{\mathfrak p} \simeq C^{\cdot}_{R_{\mathfrak p}}(IR_{\mathfrak p})
\]
provided $\height I = \height IR_{\mathfrak p}.$
\item[(d)] There is a natural isomorphism $\Hom_R(H^c_I(R), H^c_I(R)) \simeq \Ext_R^c(H^c_I(R), R).$
\end{itemize}
\end{lemma}

\begin{proof} For the proof of (a) apply the derived functor $\Rgam (\cdot)$ to the short
exact sequence as given in \ref{2.1}. Then $\Rgam (\gam (E^{\cdot})) \simeq E[-n].$ So the
long exact cohomology sequence of the corresponding exact sequence of complexes provides the
claim (cf. \cite[Lemma 2.2]{HS} for the details). The statement (b) is shown in
\cite[Corollary 2.9]{HS} and \cite[Lemma 1.2]{HS}.

For the proof of (c) localize the exact sequence of \ref{2.1} at $\mathfrak p.$ Then there
is a short exact sequence of complexes
\[
0 \to H^c_{IR_{\mathfrak p}}(R_{\mathfrak p})[-c] \to
\Gamma_{IR_{\mathfrak p}}(E^{\cdot}_{R_{\mathfrak p}})
\to C^{\cdot}_R(I) \otimes_R R_{\mathfrak p} \to 0.
\]
To this end recall first that $c = \height IR_{\mathfrak p} = \height I$ and that the local
cohomology commutes with localization. Furthermore $E^{\cdot} \otimes R_{\mathfrak p}$ is
isomorphic to the minimal injective resolution $E^{\cdot}_{R_{\mathfrak p}}$ of $R_{\mathfrak p}. $
Then the definition of the truncation complex proves the claim.

Finally, we prove (d). As it is shown at the the beginning of this section there is an exact sequence $0 \to H^c_I(R) \to \gam(E^{\cdot})^c \to \gam(E^{\cdot})^{c+1}.$ This induces a natural commutative diagram with exact rows
\[
  \begin{array}{cccccc}
    0 \to & \Hom_R(H^c_I(R), H^c_I(R)) & \to & \Hom_R(H^c_I(R),\gam(E^{\cdot}))^c  & \to  & \Hom_R(H^c_I(R),\gam(E^{\cdot}))^{c+1} \\
      &     \downarrow             &      & \downarrow                          &      & \downarrow \\
    0 \to & \Ext_R^c(H^c_I(R), R)        & \to  & \Hom_R(H^c_I(R), E^{\cdot})^c       & \to  & \Hom_R(H^c_I(R), E^{\cdot})^{c+1} \\
  \end{array}
\] 
because $\gam(E^{\cdot})$ is a subcomplex of $E^{\cdot}.$ The two last vertical homomorphisms 
are isomorphisms. This follows because $\Hom_R(X, E_R(R/\mathfrak p)) = 0$ for an $R$-module $X$ 
with $\Supp_R X \subset V(I)$ and $\mathfrak p \not\in V(I).$ Therefore the first vertical map 
is also an isomorphism. 
\end{proof}

In order to compute the local cohomology of the truncation complex $C^{\cdot}_R(I)$ there is
the following spectral sequence for the computation of the hyper cohomology of a complex.

\begin{proposition} \label{2.4} With the notation of \ref{2.1} there is the following spectral
sequence
\[
E_2^{p,q} = H^p_{\mathfrak m}(H^q(C^{\cdot}_R(I))) \Longrightarrow E^{p+q}_{\infty} =
H^{p+q}_{\mathfrak m}(C^{\cdot}_R(I)),
\]
where $H^q(C^{\cdot}_R(I)) = 0$ for $i \leq c$ and $i > n$ and $H^q(C^{\cdot}_R(I))
\simeq H^q_I(R)$ for $c < i \leq n.$
\end{proposition}

\begin{proof} The spectral sequence is a particular case for the spectral sequence of hyper
cohomology (cf. \cite{cW}). For the initial terms check the definition of the truncation complex.
\end{proof}

In the following we shall use the notion of the dimension $\dim X$ for $R$-modules $X$ with
are not necessarily finitely generated. This is defined by $\dim X = \dim \Supp_R X,$ where
the dimension of the support is understood in the Zariski topology of $\Spec R.$ In particular,
$\dim X < 0$ means $X = 0.$

\begin{lemma} \label{2.5} With the notation above we have the following results:
\begin{itemize}
\item[(a)] $\dim H^i_I(R) \leq n - i$ for all $i \geq c = \height I.$
\item[(b)] $\dim H^c_I(R) = \dim R/I.$
\item[(c)] If $\dim H^i_I(R) < n - i$ for all $i > c,$ then $R/I$ is unmixed, i.e.
$c = \height IR_{\mathfrak p}$
for all minimal $\mathfrak p \in V(I).$
\end{itemize}
\end{lemma}

\begin{proof} (a): This result is well-known (cf. for instance \cite{kK}). 

(b): Let $\mathfrak p \in V(I)$ denote a minimal prime ideal in $V(I)$ such that
$\dim R_{\mathfrak p} = c.$ Then $H^c_I(R) \otimes_R R_{\mathfrak p} \simeq
H^c_{\mathfrak p R_{\mathfrak p}}(R_{\mathfrak p}) \not= 0$ by the Grothendieck non-vanishing
result. So, $\mathfrak p \in \Supp H^c_I(R)$ and  $\dim R/\mathfrak p = d.$ Together with (a)
this proves the claim.

(c): Let $\mathfrak p \in V(I)$ minimal with $h := \height IR_{\mathfrak p} > c.$ Then
$h = \dim R_{\mathfrak p}$ and
\[
0 \not= H^h_{\mathfrak p R_{\mathfrak p}}(R_{\mathfrak p}) \simeq H^h_I(R) \otimes_R R_{\mathfrak p}.
\]
So, it
implies that $\mathfrak p \in \Supp H^h_I(R)$ with $\dim R/\mathfrak p + h = n,$ a contradiction.
\end{proof}

\begin{proof} {\bf Theorem \ref{1.2}.} Let $\alpha \geq 1$ denote an integer. The inclusion
$J \subset I$ induces a short exact sequence $0 \to I^{\alpha}/J^{\alpha} \to R/J^{\alpha}
\to R/I^{\alpha} \to 0.$ By applying the long exact cohomology sequence
with respect to $\Ext^{\cdot}_R(\cdot, R)$ and passing to the direct limit there is the following
exact sequence
\[
0 \to H^c_I(R) \to H^c_J(R) \stackrel{\phi}{\to} \varinjlim \Ext^c_R(I^{\alpha}/J^{\alpha}, R).
\]
Recall that $\grade I^{\alpha}/J^{\alpha} \geq c$ for all $\alpha.$ Let $X = \im \phi.$
The short exact sequence $0 \to H^c_I(R) \to H^c_J(R) \to X \to 0$ provides (after applying
$\Ext^{\cdot}_R(\cdot, R)$) a natural homomorphism
\[
\Ext^c_R(H^c_J(R),R) \to \Ext^c_R(H^c_I(R),R).
\]
By Lemma \ref{2.2} (d) this proves the statement in (a).

In order to prove (b) we may assume that $J R_{\mathfrak p} = I R_{\mathfrak p}$ for all
$\mathfrak p \in V(I)$ with $\dim R_{\mathfrak p} \leq c + 1.$ This follows because local
cohomology does not change by passing to the radical. Next we claim that $\dim X \leq d - 2.$
This follows because $\dim_R I^{\alpha}/J^{\alpha} \leq d-2$ for all $\alpha \in \mathbb N$ under the additional assumptions of $J \subset I.$  Moreover $\dim X \leq d - 2$ is true by a localization argument
and the embedding $X \to \varinjlim \Ext^c_R(I^{\alpha}/J^{\alpha}, R).$

By passing to the completion and because of the Matlis duality (cf. \cite[Lemma 1.2]{HS})
it will be enough to show that the natural homomorphism $H^d_{\mathfrak m}(H^c_I(R)) \to
H^d_{\mathfrak m}(H^c_J(R))$ is an isomorphism. Now this is true by virtue of the local cohomology
with respect to the maximal ideal applied to the short exact sequence $0 \to H^c_I(R) \to H^c_J(R)
\to X \to 0$ and $\dim X \leq d-2.$

For the proof of (c) recall that for a cohomologically complete intersection $J$ it is known that
the endomorphism ring of $H^c_J(R)$ is isomorphic to $R$ (cf. \cite{HSt} or \cite[Lemma 3.3]{HS}).
\end{proof}

\section{Dimensions of Local Cohomology}
As before let $(R, \mathfrak m)$ denote a $n$-dimensional
Gorenstein ring. Let $I \subset R$ be an ideal with $c = \height I$
and $\dim R/I = n -c.$ We prove the following theorem in order to estimates the dimension of local cohomology modules. To this
end let us fix the abbreviation $h(\mathfrak p) = \dim R_{\mathfrak
p} - c$ for a prime ideal $\mathfrak p \in V(I).$

\begin{theorem} \label{3.1} Let $l \geq 1$ denote an integer. With the previous notation the following conditions
are equivalent:
 \begin{itemize}
  \item[(i)] $\dim H^i_I(R) \leq n - l -i$ for all $i > c.$
  \item[(ii)] For all $\mathfrak p \in V(I)$ the natural map
\[
 H^{h(\mathfrak p)}_{\mathfrak pR_{\mathfrak p}}(H^c_{IR_{\mathfrak p}}(R_{\mathfrak p})) \to E(k(\mathfrak p))
\]
is surjective resp. bijective if $l \geq 2$ and 
\[
H^i_{\mathfrak pR_{\mathfrak p}}(H^c_{IR_{\mathfrak p}}(R_{\mathfrak p})) = 0
\]
for all $h(\mathfrak p) -l + 1 < i < h(\mathfrak p).$
 \end{itemize}
\end{theorem}

\begin{proof} (i) $\Longrightarrow$ (ii): By virtue of Lemma \ref{2.5} it follows that $R/I$ is unmixed, i.e.
$c = \height I = \height IR_{\mathfrak p}$ for all minimal prime ideals $\mathfrak p\in V(I).$ In particular
it implies that $h(\mathfrak p) = \dim R_{\mathfrak p}/I R_{\mathfrak p}$ for all prime ideals $\mathfrak p \in V(I).$
Moreover
\[
\dim R/\mathfrak p + \dim H^i_{IR_{\mathfrak p}}(R_{\mathfrak p}) \leq \dim H^i_I(R)
\]
because the localization
commutes with local cohomology. So our assumption (i) implies that $\dim H^i_{\mathfrak p R_{\mathfrak p}}(R_{\mathfrak p})
\leq \dim R_{\mathfrak p} - l -i$ for all $i > \height IR_{\mathfrak p} = c.$ Therefore it will be enough to
prove the statement in (ii) for $\mathfrak p = \mathfrak m,$ the maximal ideal of $(R, \mathfrak m).$

By virtue of Lemma \ref{2.2} (a) it will be enough to show the vanishing of $H^i_{\mathfrak m}(C^{\cdot}_R(I)$ for all
$i > n-l.$ To this end consider the spectral sequence of Proposition \ref{2.4}. By our assumption we have for the
initial terms $E_2^{p,q} = H^p_{\mathfrak m}(H^q_I(R)) = 0$ for all $p + q > n - l,$ where $q \not= c.$ This provides the
vanishing of the limit terms $H^i_{\mathfrak m}(C^{\cdot}_R(I)) = 0$ for all $i > n - l,$ as required.

(ii) $\Longrightarrow$ (i): Because of $l \geq 1$ the first statement in (ii) provides that
$H^{h(\mathfrak p)}_{\mathfrak pR_{\mathfrak p}}(H^c_{IR_{\mathfrak p}}(R_{\mathfrak p}))$ does not vanish. By virtue of
Lemma \ref{2.2} (b) it follows that $\dim R_{\mathfrak p}/IR_{\mathfrak p} \geq h(\mathfrak p)$ for all $\mathfrak p
\in V(I).$ Whence $c = \height IR_{\mathfrak p}$ for all $\mathfrak p \in V(I).$ As a consequence (cf. Lemma \ref{2.2} (c)) we see that $C^{\cdot}_R(I) \otimes_R R_{\mathfrak p} \simeq C^{\cdot}_{R_{\mathfrak p}}(IR_{\mathfrak p})$ for
all $\mathfrak p \in V(I).$

Now we proceed by induction on $d = \dim R/I.$
In the case of $d = 0$ the ideal $I$ is $\mathfrak m$-primary.
Therefore the statement is true because $R$ is a Gorenstein ring. So let $d > 0.$ First we show that the inductive hypothesis
implies
\[
\dim H^i_I(R) \leq \max \{ n-l -i, 0\} \text{ for all } c < i \leq n.
\]
To this end assume that $\dim H^i_I(R) > 0$ for a certain $i \geq n -l.$ Choose a prime ideal $\mathfrak p \in \Supp H^i_I(R) \setminus \{\mathfrak m\}.$  Therefore $H^i_{\mathfrak p R_{\mathfrak p}}(R_{\mathfrak p}) \not= 0$ and $i + l \leq \dim R_{\mathfrak p}$ by the induction hypothesis. On the other hand $l+i \leq \dim R_{\mathfrak p} < n \leq l+i$ a contradiction. Second, suppose that $\dim H^i_I(R) > n - l -i$ for a certain $c < i < n - l.$ Choose a prime ideal $\mathfrak p \in \Supp H^i_I(R)$ such that $\dim R/\mathfrak p = \dim H^i_I(R).$ Therefore $\dim R/\mathfrak p > n -l -i$ and  $l+i > \dim R_{\mathfrak p}.$ Moreover $H^i_{\mathfrak p R_{\mathfrak p}}(R_{\mathfrak p}) \not= 0$ and $i + l \leq \dim R_{\mathfrak p},$ that is again a contradiction.

With this information in mind the spectral sequence (cf. Proposition \ref{2.4}) degenerates
to isomorphisms $H^i_{\mathfrak m}(C^{\cdot}_R(I)) \simeq H^i(C^{\cdot}_R(I))$ for all $i > n-l.$ Finally the assumption in (ii)
for $\mathfrak p = \mathfrak m$ implies that $H^i_{\mathfrak m}(C^{\cdot}_R(I)) = 0$ for all $i > n-l$ (cf. Lemma \ref{2.2}). This finishes the proof
because of $H^i(C^{\cdot}_R(I)) \simeq H^i_I(R)$ for $i > c.$
\end{proof}

For $l \geq \dim R/I$ the previous result yields -- as a particular case -- the equivalence of the conditions (i) and (ii) of \cite[Theorem 3.1]{HS}. Another Corollary is the following:

\begin{corollary} \label{3.2} Suppose that $c \geq 2.$ With the above notation suppose that
\[
\widehat{R_{\mathfrak p}} \to
\Hom_{\widehat{R_{\mathfrak p}}}(H^c_{I\widehat{R_{\mathfrak p}}}(\widehat{R_{\mathfrak p}}),
H^c_{I\widehat{R_{\mathfrak p}}}(\widehat{R_{\mathfrak p}}))
 \]
is an isomorphism for all $\mathfrak p \in V(I) \setminus \{\mathfrak m\}$ (e.g. this is satisfied
in the case $I$ is locally a cohomological complete intersection). Then the following conditions are equivalent:
\begin{itemize}
\item[(i)]  $H^i_I(R) = 0$ for  $i = n-1, n.$
\item[(ii)] The natural homomorphism $\hat R \to \Hom_{\hat R}(H^c_{I \hat R}(\hat R),
H^c_{I \hat R}(\hat R))$ is an isomorphism.
\end{itemize}
\end{corollary}

\begin{proof} By the Local Duality Theorem the assumption is equivalent to the isomorphisms 
\[
H^{h(\mathfrak p)}_{\mathfrak pR_{\mathfrak p}}(H^c_{IR_{\mathfrak p}}(R_{\mathfrak p})) \to E(k(\mathfrak p))
\]
for all $\mathfrak p \in V(I)\setminus \{\mathfrak m\}.$ By a localization argument and Theorem \ref{3.1} this is equivalent to $\dim H^i_I(R) \leq \max \{n-2-i, 0\}$ for all $i > c.$ Therefore, by Theorem \ref{3.1} the statement in (ii) holds 
if and only if $H^i_I(R) = 0$ for $i = n-1, n.$
\end{proof}

Note that Corollary \ref{3.2} proves Theorem \ref{1.1} of the Introduction. Another Corollary of Theorem 
\ref{3.1} is the following vanishing result.

\begin{corollary} \label{3.3} Fix the notation as above. Suppose that $I$ is locally
a cohomological complete intersection. For an integer $l \geq 1$ the following conditions are
equivalent:
\begin{itemize}
\item[(i)] $\cd I \leq \max \{n - l, c\},$ i.e. $H^i_I(R) = 0$ for all $i > \max\{n - l, c\}.$
\item[(ii)] The natural homomorphism $H^d_{\mathfrak m}(H^c_I(R)) \to E$ is surjective resp. bijective if $l \geq 2$
and $H^i_{\mathfrak m}(H^c_I(R)) = 0$ for $d - l + 1 < i < d.$
\end{itemize}
\end{corollary}

\begin{proof} Note that the ideal $I$ is locally a cohomological complete intersection if and 
only if $\dim H^i_I(R) \leq 0$ for all $i > c.$ This follows by localization and because of $H^i_{I R_{\mathfrak p}}(R_{\mathfrak p}) = 0$ for all $i \not= c$ and all $\mathfrak p \in V(I)\setminus \{\mathfrak m\}.$ Therefore, as a consequence of Theorem \ref{3.1}, the statement is true.
\end{proof}

\section{Problems and Examples}

The first example shows that the assumptions in Corollary \ref{3.2} are not necessary for the equivalence of both of the statements. Moreover, it shows that the isomorphism $R \simeq \Hom_R(H^c_I(R), H^c_I(R))$ does not localize.

\begin{example} \label{4.2} (cf. \cite[Example 4.1]{HS}) Let $k$ be an arbitrary field. Let
$R = k[|x_0,\ldots,x_4|]$ denote the
formal power series ring in five variables over $k.$ Let
\[
I =
(x_0,x_1)\cap (x_1,x_2)\cap (x_2,x_3)\cap (x_3,x_4).
\]
Then $c = \height I = 2$ and  $H^i_I(R) = 0$ for all $i \not= 2,3,$ as it follows by the
use of the Mayer-Vietoris sequence for local cohomology. Moreover (cf. \cite[Example 4.1]{HS}) it is shown that $H^3_I(R) = E_R(R/\mathfrak p), \mathfrak p = (x_0,x_1,x_3,x_4).$ The spectral sequence
\[
E_2^{p,q} = H^p_{\mathfrak m}(H^q_I(R)) \Longrightarrow E_{\infty}^{p+q} = H_{\mathfrak m}^{p+q}(R)
\]
provides an isomorphism $H^3_{\mathfrak m}(H^2_I(R)) \simeq E.$ Recall that $H^p_{\mathfrak m}(H^3_I(R)) = 0$ for all $p \in \mathbb N.$ By Local Duality it follows that the natural
homomorphism $R \to \Hom_R(H^2_I(R), H^2_I(R))$ is an isomorphism. On the other hand, it is easily seen that this is not true for $\widehat{R_{\mathfrak p}}$ because
\[
H^2_{IR_{\mathfrak p}}(R_{\mathfrak p}) \simeq H^2_{I_1R_{\mathfrak p}}(R_{\mathfrak p}) \oplus H^2_{I_2R_{\mathfrak p}}(R_{\mathfrak p}), I_1 = (x_0,x_1), I_2 = (x_3,x_4),
\]
decomposes into two non-zero direct summands. This is seen by the use the Mayer-Vietoris sequence for local cohomology.
\end{example}

The following example shows that the endomorphism ring $\Hom_R(H^c_I(R), H^c_I(R)), c = \height I,$ is in general not a finitely generated $R$-module.

\begin{example} \label{4.3} (cf. \cite[{\S} 3]{rH2}) Let $k$ denote a field and $R = k[|x,y,u,v|]/(xu-yv),$ where
$k[|x,y,u,v|]$ denotes the power series ring in four variables over
$k.$ Let $I = (u,v)R.$ Then $\dim R = 3, \dim R/I = 2$ and $c = 1.$
It follows that $H^i_I(R) = 0$ for $i \not= 1,2.$ Moreover $\Supp H^2_I(R) \subset \{\mathfrak m\}.$ The truncation
complex with the short exact sequence (cf. \ref{2.1})
\[
0 \to H^c_I(R)[-c] \to \gam (E^{\cdot}) \to C^{\cdot}_R(I) \to 0
\]
induces a short exact sequence on local cohomology
\[
0 \to H^2_I(R) \to H^2_{\mathfrak m}(H^1_I(R)) \to E \to 0
\]
(cf. Lemma \ref{2.2}).
Hartshorne (cf. \cite[{\S} 3]{rH2}) has shown that the socle of
$H^2_I(R)$ is not a finite dimensional $k$-vector space. Therefore,
the socle of $H^2_{\mathfrak m}(H^1_I(R))$ is infinite. Moreover there are the following isomorphisms
\[
\Hom_R(H^1_I(R), H^1_I(R)) \simeq \Ext^1_R(H^1_I(R), R) \simeq \Hom_R(H^2_{\mathfrak m}(H^1_I(R)), E)
\]
(cf. Lemma \ref{2.2} (d) and \cite[Lemma 1.2]{HS}). By the Nakayama Lemma this means that $\Hom_R(H^1_I(R), H^1_I(R))$ is not a finitely generated $R$-module.
\end{example}

So, one might ask for a characterization of the finiteness of the endomorphism ring of $H^c_I(R), c = \height I.$


\end{document}